\documentclass[english,leqno,11pt]{amsart}
\usepackage{amsmath}
\usepackage{latexsym}
\usepackage[psamsfonts]{amssymb}
\usepackage[colorlinks=true, citecolor=blue]{hyperref}

\newcommand{\R}{\mathbb{R}}

\newcommand{\PP}{\mathbb{P}}
\newcommand{\C}{\mathbb{C}}
\newcommand{\Q}{\mathbb{Q}}
\newcommand{\A}{\mathcal{A}}
\newcommand{\Z}{\mathbb{Z}}

\renewcommand{\a}{\mathfrak{a}}
\renewcommand{\b}{\mathfrak{b}}

\renewcommand{\H}{\mathcal{H}}
\renewcommand{\le}{\leqslant}
\renewcommand{\ge}{\geqslant}

\newcommand{\res}{\mathrm{Res }}

\newcommand{\vol}{\mathrm{Vol}}

\newcommand{\EE}{\mathcal{E}}
\newcommand{\F}{\mathcal{F}}

\newcommand{\GL}{\mathrm{GL}}

\newcommand{\re}{\mathrm{Re}}\newcommand{\im}{\mathrm{Im}}

\newcommand{\tr}{\mathrm{Tr}}
\newcommand{\SL}{\mathrm{SL}}
\newcommand{\PSL}{\mathrm{PSL}}

\newcommand{\D}{\mathcal{D}}

\renewcommand{\SS}{\mathcal{S}}

\newcommand{\ov}[1]{\overline{#1}}
\newcommand{\og}{\Gamma}

\renewcommand{\phi}{\varphi}
\newcommand{\comment}[1]{}

\theoremstyle{plain}
\newtheorem{theorem}{Theorem}[section]
\newtheorem{corollary}[theorem]{Corollary}
\newtheorem{prop}[theorem]{Proposition}
\newtheorem{lemma}[theorem]{Lemma}
\numberwithin{equation}{section}

\theoremstyle{definition}

\title{On the Petersson scalar product of arbitrary modular forms}
\author{Vicentiu Pasol, Alexandru A. Popa}

\address{Institute of Mathematics ``Simion Stoilow" of the Romanian Academy,
P.O. Box 1-764, RO-014700 Bucharest, Romania}
\address{E-mail: vicentiu.pasol@imar.ro}
\address{E-mail: alexandru.popa@imar.ro}

\begin{document}

\begin{abstract}We consider a natural extension of the Petersson scalar product to the entire
space of modular forms of integral weight $k\ge 2$ for a finite index subgroup of the modular
group. We show that Hecke operators have the same adjoints with respect to this inner
product as for cusp forms, and we show that the Petersson product is nondegenerate for
$\Gamma_1(N)$ and $k>2$. For $k=2$ we give examples when it is degenerate, and when it is
nondegenerate. 

\end{abstract}

\maketitle

\section{Introduction}

Let $\Gamma$ be a finite index subgroup of $\SL_2(\Z)$,
let $k\ge 2$ be an integer, and denote by $M_k(\Gamma)$, $S_k(\Gamma)$ the spaces of modular
forms, respectively cusp forms of weight $k$ for $\Gamma$. For $f,g\in M_k(\Gamma)$, at least
one of which is a
cusp form, the
Petersson scalar product is defined by
\[(f,g)=\frac{1}{[\Gamma_1:\ov{\Gamma}]}\int_{\Gamma\backslash \H} f(z)\ov{g(z)} y^k
\frac{dx
dy}{y^2}.
\] 
where $\Gamma_1=\PSL_2(\Z)$ and $\ov{\Gamma}$ denotes the projectivisation of $\Gamma$. 
An extension to $M_k(\Gamma_1)$ was given by Zagier \cite{Z81}, using a renormalized integral
over a fundamental domain for $\Gamma_1\backslash \H$. In this note we use the same approach
to extend the Petersson product to all of $M_k(\Gamma)$. We show that the extended
Petersson product has the same equivariance properties under the action of Hecke operators as
the usual one, and for $\Gamma=\Gamma_1(N)$ we show that it is nondegenerate when $k>2$. When
$k=2$, we find somewhat surprisingly that it can be degenerate, and we give examples when it
is nondegenerate and when it is degenerate. 

Our motivation comes from the theory of period polynomials associated to modular forms. In
\cite{PP}, we generalize Haberland's formula by showing that the extended Petersson
product of arbitrary modular forms is given by a pairing on their (extended) period
polynomials. The nondegeneracy of the extended Petersson product is then needed to show that 
$M_k(\Gamma)$ is isomorphic to the plus and minus parts of the space of period polynomials of
all modular forms, extending the classical Eichler-Shimura isomorphism. For $\Gamma_0(N)$, the
Hecke equivariance is used to show that the pairing on extended period polynomials is also
Hecke equivariant. 

Other approaches to extending the Petersson inner product to all modular forms 
are given by Chiera \cite{C07}, and by Deitmar and Diamantis \cite{DD}. An
adelic version of the renormalization method is given by Michel and Venkatesh \cite[Ch.
4.3]{MV}. For $\Gamma_1(p)$ with $p$ prime, the nondegeneracy of the extended Petersson
product was also proved by B\"ocherer and Chiera \cite[Prop. 5.5]{BC}.

\section{Extended Petersson scalar product}

We give three equivalent definitions of the extended Petersson product following \cite{Z81}.
The first definition appears naturally in \cite{PP}, where we show that the Petersson
product of $f,g \in M_k(\og)$ can be computed in terms of a pairing on the period polynomials
of $f,g$, generalizing a formula proved by Haberland for $S_k(\Gamma_1)$. 

For the first definition assume that $\Gamma$ is a finite index subgroup of $\SL_2(\Z)$, and
let $C_\Gamma=[\Gamma_1:\ov{\Gamma}]$. Let
$\F$ be the fundamental domain $\{z\in \H: |z|\ge 1, |\re\, z |\le 1/2 \}$ for $\Gamma_1$, and
for $T>1$ let $\F_T$ be the truncated domain for which $\im\, z <T$. Let
$$\tr(f\ov{g})(z)= \sum_A f|A(z)\ov{g}|A(z),$$ 
where here and below sums over $A$ are over complete system of representatives for
$\ov{\Gamma}\backslash\Gamma_1$.\footnote{If $k$ is odd, there is a sign ambiguity in
defining $f|A$ for $A$ in $\Gamma_1=\PSL_2(Z)$, but the ambiguity dissapears when considering
the product $f|A\ov{g}|A$.} The function $y^k\tr{(f\ov{g})}$ 
is a $\Gamma_1$-invariant, renormalizable function in the sense of 
\cite{Z81}, satisfying  $y^k\tr(f\ov{g})(z)=a_0(f,g)y^k+O(y^{-K})$ for all
$K$, with $a_0(f,g)=\sum_A a_0(f|A)\overline{a_0(g|A)}$. Therefore we can define for $f,g\in
M_k(\Gamma)$ the scalar product  
\begin{equation}\label{2.1}
\begin{aligned}
 (f,g)&=\frac{1}{C_\Gamma}\lim_{T\rightarrow \infty}\Big[ \int_{\F_T} 
   y^k\tr{(f\ov{g})} d\mu - a_0(f,g)\frac{T^{k-1}}{k-1} \Big]\\
 &=\frac{1}{C_\Gamma}\int_{\F}[y^k\tr{(f\ov{g})} (z) - a_0(f,g)E(z,k)]d\mu
\end{aligned}
\end{equation}
where $E(z,s)=\sum_{\gamma \in\Gamma_{1\infty}\backslash \Gamma_1} \im(\gamma z)^s$ is the
weight 0 Eisenstein series ($\Gamma_{1\infty}$ is the stabilizer of the cusp $\infty$), and
$d\mu=\frac{dx dy}{y^2}$ is the $SL_2(\R)$-invariant measure.

The next version can be defined for an arbitrary Fuchsian group of the first kind.
Let $\D$ be a fundamental domain for $\Gamma$, and for $s\in \C$ let $\A_s(\Gamma)$ be the
space of (weight 0) automorphic  functions  for $\Gamma$, which are eigenforms of the
hyperbolic Laplacian with eigenvalue $s(s-1)$. For any function $F_{f,g}\in \A_k(\Gamma)$ such
that $f(z)\overline{g(z)} y^k-F_{f,g}(z)$ vanishes at all cusps (an example will be given
shortly), we have:                                                                           
\begin{equation}\label{2.3}
 (f,g)=\frac{\pi}{3\vol\ \D}\int_{\D}
[f(z)\overline{g(z)} y^k  - F_{f,g}(z)] d\mu
\end{equation}
The right side is independent of $F_{f,g}$: if $F'_{f,g}$ is another choice, the difference 
$F_{f,g}-F'_{f,g}\in  \A_k(\Gamma)$ is a cusp form, so its integral over $\D$ vanishes.

Assume now that $\Gamma$ is a finite index subgroup of $\SL_2(\Z)$, so that $\vol\ \D=
\frac{\pi}{3} C_\Gamma$. Decomposing $\D=\cup_{A} A \F$, \eqref{2.3} becomes        
\[
 (f,g)=\frac{1}{C_\Gamma}\int_{\F}\Big[  y^k \tr(f\ov{g})(z) - \sum_A F_{f,g}(Az)
\Big]d\mu.
\] 
Since both $\sum_A F_{f,g}(Az)$ and $a_0(f,g)E(z,k)$ belong to $\A_k(\Gamma_1) $, and they
have the same behaviour at infinity as $y^k \tr(f\ov{g})$ it follows as before that the last
equation agrees with \eqref{2.1}.

To give an example of $F_{f,g}$ as above, let $\SS\subset \PP^1(\Q)$ be a complete set of
inequivalent cusps of $\Gamma$. For $\a\in \PP^1(\Q)$ fix $\sigma_\a\in PSL_2(\R)$ with
$\sigma_\a \infty=\a$. Let $\Gamma_\a$ be the subgroup of
$\Gamma$ of elements fixing $\a$. Define the weight 0 Eisenstein series associated with the
cusp $\a$ by:
\[
E_\Gamma^\a(z,s)=\sum_{\gamma \in\Gamma_{\a}\backslash \Gamma} \im(\sigma_\a^{-1}\gamma z)^s 
\]
which converges absolutely for $\re s > 1 $, and belongs to  $\A_s(\Gamma)$. From the
Fourier expansion of $E_\Gamma^\a(\sigma_\b z,s)$ (see \cite{Iw02}, Theorem 3.4), it follows
that 
\[
E_\Gamma^\a(\sigma_\b z,s)=\delta_{\a\b} y^s+\phi_{\a\b}(s)y^{1-s}+O\big((1+y^{-\re s})e^{-2\pi
y} \big)
\] 
with $\delta_{\a\b}=1$ if $\a=\b$ and 0 otherwise, and $\phi_{\a\b}(s)$ an explicit function. 
Assuming that the fundamental domain $\D$ is chosen such that its vertices on the boundary of
$\H$ are precisely a complete set of representatives for the cusps of $\Gamma$, it follows
that the linear combination $$F_{f,g}(z)=\sum_{\a \in \SS}
a_0(f|{\sigma_\a})a_0(\ov{g|\sigma_{\a}})
E_\Gamma^\a(z,k)\in \A_k(\Gamma) $$ is such that $f(z)\overline{g(z)} y^k -F_{f,g}(z)$ vanishes
at all cusps, so $F_{f,g}$ is a valid choice in \eqref{2.3}.

Lastly, assuming $\Gamma$ is a finite index subgroup of $\SL_2(\Z)$, from \eqref{2.3} with the
choice of $F_{f,g}$ as in the previous paragraph we have
\begin{equation}\label{2.2}
 (f,g)=\frac{\pi}{3}(4\pi)^{-k}\Gamma(k)\res_{s=k}L(s,f,\ov{g}).
\end{equation}
This identity is well known if $f,g$ are cuspidal when it goes back to
Rankin. If both $f,g$ are noncuspidal, it follows from extending to $\Gamma$ the Rankin-Selberg
method developed in \cite{Z81} for the full modular group. Since the generalization is
straightforward, we omit the details.

\section{Adjoints of Hecke operators}

In this section we show that Hecke operators have the same adjoints with respect to the
extended Petersson product on $M_k(\Gamma)$ as with respect to the one on $S_k(\Gamma)$. The
proof copies the classical one given in \cite{Sh}, using the definition of
$(f,g)$ in \eqref{2.3}. We assume $\Gamma$ is a Fuchsian subgroup of $\SL_2(\R)$ of the first
kind, namely a subgroup acting discretely on $\H$ and of finite covolume.

Let $\tilde{\Gamma}$ consist of elements $\alpha$ of $\GL_2(\R)$ such that
$\alpha\Gamma \alpha^{-1}$ is commensurable with $\Gamma$. For $\alpha\in
\tilde{\Gamma}$, let $\Gamma=\cup_{i=1}^r (\Gamma\cap \alpha^{-1}\Gamma \alpha) \gamma_i$
(disjoint union). Then $\alpha \gamma_i$ is a complet system of representatives for
$\Gamma\backslash  \Gamma \alpha \Gamma$ and the action of the Hecke operator associated with
the coset $\Gamma \alpha \Gamma$ is defined on $f\in M_k(\Gamma)$ by
\[
f|[\Gamma \alpha \Gamma] =n^{k-1} \sum_{i=1}^r f|_{k} \alpha \gamma_i
\]
where $n=\det \alpha$ and $f|_{k} \gamma (z)= f(\gamma z) j(\gamma, z)^{-k}$ for $\gamma\in
\GL_2(\R)$ (note that the stroke operator is normalized differently than by Shimura). 

\begin{prop}\label{p1} The adjoint of the operator $[\Gamma \alpha \Gamma]$ is $[\Gamma
\alpha^\vee \Gamma]$ namely for $f,g\in M_k(\Gamma)$ 
\[ 
(f|[\Gamma \alpha \Gamma] ,g)=(f, g|[\Gamma \alpha^\vee\Gamma]) 
\]
where $\alpha^\vee=\alpha^{-1} \det \alpha$. 
\end{prop}
\begin{proof} The proof is identical to that of eq. (3.4.5) in \cite{Sh}, except that a
term involving $F_{f,g}$ has to be subtracted at each step, and one has to use repeatedly the
fact that $F_{f,g}$ can be replaced by any other function in $\A_k(\Gamma)$ having the same
behaviour at the cusps. 
\end{proof}

\section{Nondegeneracy}

We show that the Petersson product is
nondegenerate on $M_k(\Gamma)$ for $\Gamma=\Gamma_1(N)$ and for $k> 2$. For $k=2$ and 
$\Gamma=\Gamma_0(N)$ or  $\Gamma=\Gamma_1(N)$, we give examples when it is degenerate or
nondegerate. For the proof, we compute explicitly the
determinant of the matrix of the Petersson product, with respect to a basis of Hecke eigenforms
for $\EE_k(\Gamma)$, using formula \eqref{2.2}. The degeneracy when $k=2$ is somewhat
surprising, and we give an alternate proof  for $\Gamma=\Gamma_0(6)$ in \cite[Sec. 7]{PP},
using period polynomials and a generalization of Haberland's formula.

Let $\psi, \phi$ be primitive characters of conductors $c_\psi, c_\phi$ with $c_\psi c_\phi|N$
and $\psi\phi(-1)=(-1)^k$, and let $t$ be a
divisor of $N/(c_\psi c_\phi )$. When $k>2$, a basis of $\EE_k(\Gamma)$ consists of the
Eisenstein series $E_k^{\psi,\phi, t}(z)=E_k^{\psi,\phi}(tz)$ where
\[
E_k^{\psi,\phi}(z)=\frac{\delta(\psi)}{2} L(1-k, \phi) + \sum_{n=1}^\infty
\sigma_{k-1}^{\psi,\phi}(n) q^n 
\]
where  $\sigma_{k-1}^{\psi,\phi}(n)=\sum_{n=ad} \psi(a) \phi(d) d^{k-1}$, and $\delta(\psi)$
is 1 if $\psi={\bf 1}$ (the character of
conductor 1), and zero otherwise. When $k=2$, the same elements form a basis, with the
series $E_k^{{\bf 1},{\bf 1}, t}$ replaced by
$E_k^{{\bf 1},{\bf 1}}(z)-tE_k^{{\bf 1},{\bf 1}}(tz)$ for $t>1$, and  with $E_k^{{\bf 1},{\bf
1}, 1}$ removed. For $\chi$ a character mod $N$, a basis of the space of
Eisenstein series for $\Gamma_0(N)$ with character $\chi$ consists of those $E_k^{\psi,\phi,
t}$ for which $\chi=\psi\phi$. These Eisenstein series are
Hecke eigenforms for the operators $T_n$ with $(n,N)=1$ \cite{DS}.

For $(n,N)=1$, the adjoint of $T_n$ with respect to the Petersson scalar product is the
operator $T_n^*=<n>^{-1}T_n$, with $<n>$ the diamond operator. On the basis above $<n>$ acts by
\[
<n> E_k^{\psi,\phi, t}= \psi(n)\phi(n)E_k^{\psi,\phi, t},
\]
and Proposition \ref{p1} shows that $  \big(E_k^{\psi,\phi, t},
E_k^{\psi',\phi', t'}\big)=0$ unless $\psi=\phi', \phi=\psi'$. Therefore
the Petersson product is nondegenerate on $\EE_k(\Gamma)$ if and only if for every 
pair $\psi, \phi$ as above, the matrix $$M_{\psi,\phi}= [ ( E_k^{\psi,\phi, t}, E_k^{\phi,\psi,
t'}) ]_{t,t'}$$  is nonsingular, where the rows and columns are indexed by divisors $t,t'$ of
$N/(c_\psi c_\phi)$ (with $t,t'\ne 1$ if $k=2$ and $\psi=\phi={\bf 1}$).

We compute the entries of $M_{\psi,\phi}$ with the aid of 
\eqref{2.2}. Assuming $(k,\psi, \phi)\ne (2,{\bf 1}, {\bf 1})$, we
have $L(s,E_k^{\psi,\phi})=L(s,\psi)L(s-k+1, \phi)$, which has an Euler product. Using
\cite[Lemma 1]{Sh1} we get
\begin{multline*}
L(s, E_k^{\psi,\phi,t}, \ov{E_k^{\phi,\psi,t'}})=
\frac{L(s,\psi\ov{\phi})L(s-2k+2, \phi\ov{\psi})L(s-k+1, \psi\ov{\psi})  L(s-k+1,
\phi\ov{\phi})}{L(2s-2k+2, \psi\phi\ov{\psi}\ov{\phi})} \cdot \\
\cdot (drr')^{-s} \prod_{p^e||rr'}
\frac{X_p(e,s)}{1-\psi\phi\ov{\psi}\ov{\phi}(p)p^{2k-2-2s}} \quad \quad \quad \ 
\end{multline*}
where $d=(t,t')$ and $t=dr$, $t'=dr'$, and $X_p(e,s)$ is a polynomial of degree $\le 2$ in
$p^{-s}$, given below. The product is over primes $p|rr'$, with $p^e|rr', p^{e+1}\nmid
rr'$, and it equals 1 if $rr'=1$.  Note that $L(s, \psi\ov{\psi})$,  $L(s,
\phi\ov{\phi})$ have simple poles at $s=1$, and $L(s, \phi\ov{\psi})$ has a zero at 
$s=2-k$. Denoting by $R_{\psi,\phi}$ the residue at $s=k$ of
the fraction on the first line, it follows from \eqref{2.2} that 
\begin{equation}\label{4.1}
M_{\psi,\phi}=\frac{\pi}{3}(4\pi)^{-k}\Gamma(k) R_{\psi,\phi} M_k^{\psi,\phi}(L)
\end{equation}
where $L=N/(c_\psi c_\phi)$ and $M_s^{\psi,\phi}(L)$ is the matrix whose rows and columns are
indexed by divisors $t,t'$ of $L$ with the entry corresponding to $t,t'$ equal to 
\begin{equation}\label{3.1}
m(t,t')=(drr')^{-s} \prod_{p^e\|rr'}
\frac{X_p(e,s)}{1-\psi\phi\ov{\psi'}\ov{\phi'}(p)p^{2k-2-2s}}.
\end{equation}

When $k>2$ we have $R_{\psi,\phi} \ne 0$, since $L(s, \phi\ov{\psi})$ has a simple zero at 
$s=2-k$ (recall $\psi\phi(-1)=(-1)^k$). Therefore the Petersson product is nondegenerate if
and only if the matrix $M_k^{\psi,\phi}(L)$ is nonsingular for all choices $\psi, \phi$ as
above. When $k=2$, one can have $R_{\psi,\phi} = 0$, as $L(s, \phi\ov{\psi})$ may have a
zero of order at least two at $s=0$ when $\phi\ov{\psi}$ is not primitive. We discuss in more
detail the case $k=2$ at the end of
this section, and we now proceed to compute $\det M_s^{\psi,\phi}(L)$, assuming only  $(k,\psi,
\phi)\ne (2,{\bf 1}, {\bf 1})$. We fix $\psi, \phi$ and let $M_s(L)=M_s^{\psi,\phi}(L)$ for
brevity. 

Let  $\alpha=\psi(p),\ \alpha'=\phi(p)p^{k-1},$ and $ \beta=\ov{\phi}(p),\
\beta'=\ov{\psi}(p)p^{k-1}$ be the local factors in the Euler product of 
$L(s, E_k^{\psi,\phi})$, and $L(s, \ov{E_k^{\phi,\psi}})$ respectively, and 
$$ a(p^n)=\frac{\alpha^{n+1}-\alpha'^{n+1}}{\alpha-\alpha'},\
b(p^n)=\frac{\beta^{n+1}-\beta'^{n+1}}{\beta-\beta'}$$
(if $\alpha=\alpha'=0$, then $a(p^n)=0$, $n>0$). If $p^e\|r$ we have
\begin{equation}\label{4.0}
X_p(e,s)=a(p^e)-a(p^{e-1}) b(p) \alpha
\alpha'p^{-s}+a(p^{e-2})(\alpha\alpha')^2\beta\beta'p^{-2s}
\end{equation}
while if $p^e\|r'$ interchange $\alpha, \alpha'$ with $\beta, \beta'$ and $a(p^i)$ with
$b(p^i)$ in the definition of $X_p(e,s)$. We use the convention $a(n)=0$ if $n\notin \Z$.

The next lemma reduces the computation of $\det M_s(L)$ to the case $L$ is a prime power. 
\begin{lemma}\label{l1} With the notations as above, consider $L_1,L_2$ two relatively prime
numbers. Then:
$$
\det(M_s(L_1L_2))=\det(M_s(L_1))^{\sigma_0(L_2)}\cdot\det(M_s(L_2))^{\sigma_0(L_1)},
$$
where $\sigma_0(L)$ denotes the number of divisors of $L$. \end{lemma}

\begin{proof} If $t_1 t_2$, $t_1't_2'$ are two divisors of $L_1L_2$, with
$t_i , t_i'|L_i$, it follows from \eqref{3.1} that 
$m(t_1t_2, t_1't_2')=m(t_1,t_1') m(t_2, t_2'),$  
so the matrix $M(L_1L_2)$ is the Kronecker product of the matrices $M(L_1)$, $M(L_2)$ and the
conclusion follows.
\end{proof}

\begin{lemma}\label{l2} For $p$ a prime and $n\ge 1$, let $C_{p,s}= p^{\frac{n(n+1)}{2} s} $
and $y=p^{k-1-s}$. We have: 
$$
C_{p,s}\det M_s(p^n)=\begin{cases} 1 & \text{ if } \phi(p)=\psi(p)=0 \\
	  (1-y)^n & \text{ if exactly one of } \phi(p), \psi(p) \text{ is } 0  \\
         \frac{(1-y)^{n-1}}{(1+y)^{n+1}}\big(1-\frac{\alpha}{\alpha'}y  \big)^n
\big(1-\frac{\alpha'}{\alpha}y  \big)^n & \text{ if } \psi(p)\phi(p)\ne 0  
            \end{cases}
$$  
with $\alpha=\psi(p)$, $\alpha'=\phi(p)p^{k-1}$.  
\end{lemma}
\begin{proof} For $0\le i,j\le n$, denote by $m(i,j)=m(p^i, p^j)$ in \eqref{3.1}. 
The matrix  $M_s(p^n)$ has elements $1, p^{-s}, \ldots, p^{-ns}$ on the
diagonal. We rescale it by multiplying the $i$-th line by $p^{is}$, $0\le i\le n$, which
multiplies its determinant by $C_{p,s}$, and we denote by  $M_s'(p^n)$ the resulting matrix,
having $\det M_s'(p^n)=C_{p,s}\det M_s(p^n)$. Note that the matrix $M_s'(p^n)$ has constant
entries on all diagonals parallel to the main diagonal, and $m'(i,i)=1$, where $m'(i,j)$ are
its entries, $0\le i,j\le n$.

When $\phi(p)=\psi(p)=0$, the matrix $M_s'(p^n)$ is the identity, and the first formula
follows. 

When exactly one of $\psi(p), \phi(p)$ is 0, the off diagonal elements are
$m'(i+e,i)=a^e$, $m'(i,i+e)=b^e p^{-se}$, $e\ge 0$, with $a=a(p), b=b(p)$ as in \eqref{4.0},
$ab=p^{k-1}$. The determinant is easy to compute, by subtracting from line $i$ the quantity
$a$ times the previous line, starting with $i=n,n-1,\ldots 1$. The resulting matrix will be
diagonal, of determinant $(1-abp^{-s})^n=(1-y)^n$.  
 
Assume now that $\psi(p)\phi(p)\ne 0$. By \eqref{4.0}, for $e\ge 1$ we have $m'(i+e, i)=X(e)$,
$m'(i, i+e)=p^{-s e}Y(e)$, where $X(e)$ is given by
\[
X(e)=\frac{a(p^e)-a(p^{e-1}) a(p) y+a(p^{e-2})\alpha \alpha' y^2}{1-y^2}, \quad e\ge 1,
\]
with $a(p^{-1})=0,$ and $Y(e)$ given by the same formula as $X(e)$ with $a$ interchanged with
$b$ and $\alpha, \alpha'$ with $\beta, \beta'$.  We set $X(0)=1$, so that $m'(i+e,
i)=X(e)$ for $e\ge 0$. 

Since $\{a(p^e)\}$ satisfies the recurrence $a(p^e)-a(p^{e-1}) a(p) +a(p^{e-2})\alpha
\alpha'=0$, the same is true about $X(e)$, namely 
\[
X(e)-X(e-1) a(p) +X(e-2)\alpha \alpha'=0, \quad e\ge 3.
\] 
In fact one checks that the recurrence holds for $e=2$ as well, with $X(0)=1$, and also for
$e=1$, with $X(-1)=p^{-s}Y(1)$ (using $\alpha\alpha' b(p)=p^{k-1} a(p)$ and recalling
$y=p^{k-1-s}$). Starting with $i=n, n-1, \ldots, 2$, we subtract from the $i$-th line a
multiple $a(p)$ of the $(i-1)$-th line, and add a multiple $\alpha \alpha'$ of the $(i-2)$-th
line. The resulting matrix will be upper-triangular, except for the entry $m'(1,0)=X(1)$, so
 $$\det M_s'(p^n) = [1-X(1)Y(1)p^{-s}][1-a(p)Y(1)p^{-s}+\alpha\alpha' Y(2)p^{-2s}]^{n-1},$$
which is easily seen to equal the expression in the statement. 
\end{proof}

\begin{corollary} \label{c1}
a)  If $k>2$ we have $\det M_k(L)\ne 0$. 

b) If $k=2$ we have $\det M_k(L)= 0$ if and only if $\phi(p)=\psi(p)\ne 0$.
\end{corollary}
From the discussion above we conclude: 
\begin{theorem} Let $\Gamma=\Gamma_1(N)$. 

a) If $k>2$ then the extended Petersson product on $M_k(\Gamma)$ is nondegenerate. 

b) The extended Petersson product on $M_2(\Gamma)$ is nondegenerate if $N$ is prime. It is
degenerate: if $N$ is divisible by $p^2q$ with $p\ne q $ primes; or if $N$ is divisible by
$pq$ with $p \ne q$ primes such that $q$ is not a primitive residue mod $p$.
\end{theorem}

\begin{proof} Part a) was already proved above. 

For part b), assume $p^2q|N$. We take $\psi=\phi$ characters of conductor $p$. Then 
$\psi(q)=\phi(q)\ne 0$, and Corollary \ref{c1} shows that $\det
M_2^{\psi,\phi}(N/p^2)=0$, so the Petersson product is degenerate. 

Assuming $pq|N$ with $p,q$ as in the statement, it follows that there is a primitive
character  $\psi$ mod $p$ with $\psi(q)=1$. Taking $\phi={\bf 1}$, Corollary \ref{c1} shows
that $\det M_2^{\psi,\phi}(N/p)=0$, so the Petersson product is degenerate. 

When $N=p$ is prime, the $L$-function $L(s, E_2^{{\bf 1}, {\bf
1},p})=\zeta(s)\zeta(s-1)(1-\frac{1}{p^{s-1}})$ has an Euler product. Then
$\res_{s=k}L(s, E_2^{{\bf 1}, {\bf 1},p}, \ov{E_2^{{\bf 1}, {\bf 1},p} })$ can be
computed as before, and it is nonzero. Also if $\psi$ is a primitive character of conductor
$p$, $R_{\psi, {\bf 1}}\ne 0$ with the notation of \eqref{4.1}, so the Petersson product is
nondegenerate in this case.
\end{proof}

Note that part a) implies that the Petersson product is nondegenerate on
$M_k(\Gamma_0(N))$ for $k>2$. To investigate what happens for $k=2$ and $\Gamma=\Gamma_0(N)$,
we now consider the case $k=2, \psi=\phi={\bf 1}$. Denote $E_2=E_2^{{\bf 1}, {\bf 1}}$, and
for $t>1$ let $E_2^t(z)=E_2(z)-tE_2(tz)$. We have 
\[L(s,E_2)=\sum_{n\ge 1} \frac{a(n)}{n^s}=\zeta(s) \zeta(s-1), 
\] 
with $a(p)=1+p$ for $p$ prime, and $L(s, E_2^t)=\sum_{n\ge 1} \frac{a(n)-t a(n/t)}{n^s}$. It
follows that  $L(s, E_2^t, E_2^t)$ is a sum of four Rankin $L$-functions with an Euler
product, and we have as before
\begin{equation}\label{4.4}
L(s, E_2^t, E_2^{t'})=\frac{\zeta(s)\zeta(s-1)^2 \zeta(2-s)}{\zeta(2s-2)} \cdot m_s(t,t')
\end{equation}
where, after writing $t=dr, t=dr'$ with $d=(t,t')$, $y_p=p^{1-s}$ , we have 
\[
m_s(t,t')=1+ \frac{tt'}{(drr')^{s}} \prod_{p^e||rr'}
\frac{X_p(e,s)}{1-y_p^2} - t^{1-s}  \prod_{p^e||t}
\frac{X_p(e,s)}{1-y_p^2} - t'^{1-s}  \prod_{p^e||t'}
\frac{X_p(e,s)}{1-y_p^2}
\]
with $X_p(e,s)$ as in \eqref{4.0} with $\alpha=\beta=1, \alpha'=\beta'=p$.

\begin{theorem} Let $\Gamma=\Gamma_0(N)$ with $N>1$ square-free. Then the Petersson product is
degenerate on $M_2(\Gamma)$, unless $N$ is prime when it is nondegenerate. 
\end{theorem}
\begin{proof}
A basis for the space $\EE_2(\Gamma)$ consists of the Eisenstein series $E_2^{\psi, \phi, t}$
with $c_\phi=c_\psi=c$, $\psi\phi=1_c$ (the principal character of conductor $c$), and
$t|(N/c^2)$, so when $N$ is square-free, only the case $\psi=\phi={\bf 1}$ is possible. Since
$\frac{X_p(1,s)}{1-y_p^2}=p \gamma_s(p)$, with $\gamma_s(p)= \frac{1+p^{-1}}{1+p^{1-s}}$,
setting $\gamma_s(u)=\prod_{p|u}\gamma_s(p)$ for $u$ square-free and $\gamma_s(1)=1$, we have
\[
m_s(t,t')=1+(drr')^{2-s}\gamma_s(r)\gamma_{s}(r') - (dr)^{2-s} \gamma_s(r)\gamma_s(d)-
(dr')^{2-s} \gamma_{s}(r')\gamma_s(d)
\]
Since $\gamma_2(p)=1$ for every $p$, it follows that $m_2(t,t')=0$, so the $L$-function
\eqref{4.4} has a simple pole at $s=2$ with residue equal to $\zeta(0)$ times the quantitity
$$m'(t,t')=\frac{dm_s(t,t')}{ds}\big|_{s=2}.$$ Therefore the Petersson product is nondegerate
if and only if the matrix $M(N)$, with entries $m'(t,t')$ indexed by
divisors $t, t'$ of $N$ with $t,t'>1$, is nonsingular. 

Since $\frac{d\gamma_s(p)}{ds}\big|_{s=2}=\frac{\ln(p)}{1+p}$ and $\gamma_2(p)=1$, we have
$\frac{d\gamma_s(u)}{ds}\big|_{s=2}=\sum_{p|u} \frac{\ln(p)}{1+p}$ for $u$ squarefree, and
\[
m'(t,t')=\begin{cases}
          0 & \text{ if }  (t,t')=1 \\
			 \sum_{p|d} \frac{p-1}{p+1}\ln(p) &  \text{ if } d=(t,t')>1. 
         \end{cases}
\]
Then the lines indexed by primes $p\ne q$ add up to the line indexed by $pq$, showing that the
determinant is 0, so the pairing is degenerate unless $N=p$ is prime.
\end{proof}

\noindent\textbf{\small Acknowledgments.} The first author was partially supported by the
CNCSIS grant PD-171/28.07.2010. The second author was partially supported by the European
Community grant PIRG05-GA-2009-248569.

\end{document}